\documentclass[11pt]{amsart}

\usepackage{amscd,amsmath,amssymb,amsfonts,amsthm,mathrsfs}
\usepackage[a4paper,text={160mm,240mm},centering,headsep=5mm,footskip=10mm]{geometry}
\usepackage[all]{xy}

\usepackage[colorlinks,linkcolor=black,anchorcolor=black,citecolor=black]{hyperref}

\usepackage{array}
\usepackage{color}

\numberwithin{equation}{section}
\newtheorem{theorem}{Theorem}[section]
\newtheorem{lemma}[theorem]{Lemma}
\newtheorem{keylemma}[theorem]{Key lemma}
\newtheorem{proposition}[theorem]{Proposition}
\newtheorem{proposition/definition}[theorem]{Proposition/Definition}
\newtheorem{corollary}[theorem]{Corollary}
\newtheorem{definition}[theorem]{Definition}
\newtheorem{remark}[theorem]{Remark}

\newtheorem*{notat*}{Notation}

\newtheorem{conj}[theorem]{Conjecture}
\newtheorem{quest}[theorem]{Question}

\newcommand{\Z}{\mathbb{Z}}
\newcommand{\Q}{\mathbb{Q}}
\newcommand{\C}{\mathbb{C}}
\renewcommand{\O}{\mathcal{O}}
\renewcommand{\r}{\rightarrow}

\newcommand{\et}{\mathrm{\acute{e}t}}
\newcommand{\Spec}{\mathrm{Spec}}

\newcommand{\CH}{\mathrm{CH}}
\newcommand{\Pic}{\mathrm{Pic}}

\renewcommand{\exp}{\mathrm{exp}}

\title{Deformation theory of the Chow group of zero-cycles}
\author{Morten L\"uders}
\address{Fakult\"at f\"ur Mathematik, Universit\"at Regensburg, 93040 Regensburg, Germany}
\email{mortenlueders@yahoo.de}

\begin{document}
\begin{abstract}
We study the deformations of the Chow group of zero-cycles of the special fibre of a smooth scheme over a henselian discrete valuation ring. Our main tools are Bloch's formula and differential forms. As a corollary we get an algebraization theorem for thickened zero-cycles previously obtained using idelic techniques. In the course of the proof we develop moving lemmata and Lefschetz theorems for cohomology groups with coefficients in differential forms.
\end{abstract}

\thanks{The author is supported by the DFG through CRC 1085 \textit{Higher Invariants} (Universit\"at Regensburg). }
\maketitle

\section{Introduction}
Let $A$ be a henselian discrete valuation ring with uniformising parameter $\pi$ and residue field $k$. Let $X$ be a smooth projective scheme over Spec$(A)$ of relative dimension $d$. Let $X_n:=X\times_A A/(\pi^n)$, i.e. $X_1$ is the special fiber and the $X_n$ are the respective thickenings of $X_1$. Let $\mathcal{K}^M_{*,X}$ (resp. $\mathcal{K}^M_{*,X_n}$) be the improved Milnor K-sheaf defined in \cite{Ke09}.

Throughout this article, unless stated otherwise, we assume that either (1) $k$ is of characteristic $0$ and $A=k[[\pi]]$ or that (2)  $A$ is the Witt ring $W(k)$ of a perfect field $k$ of ch$(k)>2$. In each of these two cases there exists a well-defined exponential map
$$\exp:\Omega^{d-1}_{X_1}\r \mathcal{K}^M_{d,X_n}$$
defined by 
$$xd\text{log}(y_1)\wedge ...\wedge d\text{log}(y_{d-1})\mapsto \{1+x\pi^{n-1},y_1,...,y_{d-1}\}$$
(see \cite[Sec. 2]{BEK14'} for (1) and \cite[Sec. 12]{BEK14} for (2)).
In these two cases we therefore get an exact sequence of sheaves
$$\Omega^{d-1}_{X_1}\rightarrow \mathcal{K}^M_{d,X_n}\rightarrow \mathcal{K}^M_{d,X_{n-1}}\rightarrow 0.$$
This exact sequence induces an exact sequence
$$H^d(X_1,\Omega^{d-1}_{X_1})\rightarrow H^d(X_1,\mathcal{K}^M_{d,X_n})\rightarrow H^d(X_1,\mathcal{K}^M_{d,X_{n-1}})\rightarrow 0$$
which we use to study the restriction map
$$\begin{xy}
  \xymatrix{
       res_{X_n}: \CH^{d}(X) \ar[r]^-{\cong}   &  
    H^d(X,\mathcal{K}^M_{d,X}) \ar[r]^-{res_{X_n}} &  H^d(X_1,\mathcal{K}^M_{d,X_n}),
  }
\end{xy} $$
assuming the Gersten conjecture for the Milnor K-sheaf $\mathcal{K}^M_{*,X}$ for the isomorphism on the left. The Gersten conjecture holds in case (1) by \cite{Ke09} and in case (2) with finite coefficients if we assume ch$(k)$ to be large enough by \cite[Prop. 1.1]{Lu17'}. Our main theorem is the following:

\begin{theorem}[Cor. \ref{cormain}]\label{maintheorem}
With the above notation, and assuming the Gersten conjecture for the Milnor K-sheaf $\mathcal{K}^M_{*,X}$, the map $$res_{X_n}:\CH^d(X)=\CH_1(X)\rightarrow H^d(X_1,\mathcal{K}^M_{d,X_n})$$ is surjective. In particular the map
$res:\CH_1(X)\rightarrow "\mathrm{lim}_n"H^d(X_1,\mathcal{K}^M_{d,X_n})$
is an epimorphism in $\text{pro-}\text{Ab}$ (for the definition see Section \ref{reldim1}).
\end{theorem}

Theorem \ref{maintheorem} was also proved at the same time in \cite{Lu17'} using an idelic method. Let us recall the strategy of \cite{Lu17'} in order to distinguish it from the approach taken in this article. In \textit{loc. cit.} we consider the short exact sequence
$$0\r \mathcal{K}^M_{d,X\mid X_n} \r \mathcal{K}^M_{d,X} \r \mathcal{K}^M_{d,X_n} \r 0.$$
Then, using the mentioned idelic method, we show that $H^{d+1}(X,\mathcal{K}^M_{d,X\mid X_n})=0$ which implies the theorem.

In this article we first prove Theorem \ref{maintheorem} for relative surfaces (i.e. $d=2$) by constructing lifts from $H^2(X_1,\Omega^{1}_{X_1})$ to $\CH^2(X)$. This requires moving elements of $H^2(X_1,\Omega^{1}_{X_1})$ into good position (see Key lemma \ref{approximationlemma}).
We then deduce the general case using the following Lefschetz theorem: 
\begin{proposition}[Lemma \ref{gysinmap}, Proposition \ref{lefschetz}] Let $Y_1$ be a smooth hypersurface section of $X_1$ and $d=\mathrm{dim} X_1$. Then there is a map 
$$gys:H^{d-1}(Y_1,\Omega^{d-2}_{Y_1})\rightarrow H^d(X_1,\Omega^{d-1}_{X_1})$$ 
which is an isomorphism for $d\geq 4$ and surjective for $d=3$ if $Y_1$ is of high degree.
\end{proposition} 

On the way we also prove the following Lefschetz theorem:
\begin{proposition}[Proposition \ref{lefschetzinjective}] Let $Y_1$ be a smooth hypersurface section of $X_1$ and $d=\mathrm{dim} X_1$. Let $i$ denote the inclusion $Y_1\hookrightarrow X_1$. Then the map $$i^*:H^{q}(X_1,\Omega^{p}_{X_1})\rightarrow H^q(Y_1,\Omega^{p}_{Y_1})$$ is an isomorphism for $p+q<d-1$ and injective for $p+q=d-1$ if $Y_1$ is of high degree.
\end{proposition} 

The last two propositions might be well known to the expert (see also Remark \ref{remarkDI}).

\begin{remark}
One classically studies the formal deformations of Chow groups of arbitrary dimensional cycles via Bloch's formula setting $\CH^p(X_n):=H^p(X_1,\mathcal{K}^M_{p,X_n})$. In characteristic $0$ one then uses the commutative diagram with exact rows
$$\begin{xy}
  \xymatrix{
      H^p(X_1,\mathcal{K}^M_{p,X_{n+1}})   \ar[d]_{} \ar[r]^{} & H^p(X_1,\mathcal{K}^M_{p,X_n})  \ar[d] \ar[r] & H^{p+1}(X_1,\Omega^{p-1}_{X_1}) \ar[d] \\
    H^{p}(X_1,\Omega^{p}_{X_{n+1}}) \ar[r] & H^{p}(X_1,\Omega^{p}_{X_n}) \ar[r] & H^{p+1}(X_1,\Omega^{p-1}_{X_1})\oplus H^{p+1}(X_1,\Omega^{p}_{X_1})
  }
\end{xy} $$
induced by the commutative diagram with exact rows
$$\begin{xy}
  \xymatrix{
     0 \ar[r] &  \Omega^{p-1}_{X_1} \ar[d]_{} \ar[r]^{} &  \mathcal{K}^M_{p,X_{n+1}} \ar[d] \ar[r] & \mathcal{K}^M_{p,X_n} \ar[d]  \ar[r] & 0 \\
     0 \ar[r] &  \Omega^{p-1}_{X_1}\oplus \Omega^p_{X_1} \ar[r] &  \Omega^p_{X_{n+1}} \ar[r] &  \Omega^p_{X_n}  \ar[r] & 0
  }
\end{xy} $$
to find Hodge-theoretic conditions for the (formal) lifting of cycles. For more details on this approach see \cite{GG04} and \cite{BEK14'}. If $p=d$, i.e. in the case of zero-cycles, these conditions are vacuous for dimensional reasons. One can always lift a (thickened) zero-cycle to the next thickening. We may therefore concentrate on studying the algebraization properties of such thickened zero-cycles. As mentioned above, we do this through a detailed study of the group $H^{d}(X_1,\Omega^{d-1}_{X_1})$ which describes the difference between $H^d(X_1,\mathcal{K}^M_{d,X_{n+1}})$ and $H^d(X_1,\mathcal{K}^M_{d,X_{n}})$.
\end{remark}

The importance of Theorem \ref{maintheorem} comes from its relation to the following question of Colliot-Th\'el\`ene about the structure of the Chow group of zero cycles over $p$-adic fields. 
\begin{quest}\label{questCT}(\cite[Question 1.4(g)]{Co95})
Let $X_K$ be a smooth projective and connected variety over a $p$-adic field $K$. 
Let $A_0(X_K)$ denote the kernel of the degree map $\mathrm{deg}:\CH_0(X_K)\r \Z$ and $D(X_K)$ its maximal divisible subgroup. Is 
$$A_0(X_K)/D(X_K)\cong \Z^n_p\oplus (\mathrm{finite\; group})$$
for some $n\in \mathbb{N}$?
\end{quest}
In \cite[Sec. 10]{KEW16} Kerz, Esnault and Wittenberg conjecture that the induced map $$res:\CH_1(X)/p^r\rightarrow "\mathrm{lim}_n"H^d(X_1,\mathcal{K}^M_{d,X_n}/p^r)$$ is an isomorphism and note that this conjecture should imply a positive answer to Question \ref{questCT}. In Section \ref{sectionopenproblems} we explain how our approach in this article is related to the Question \ref{questCT} and the injectivity of the restriction map $res$ in the above conjecture.

Finally, in \cite{Lu17'} it is shown that for $A=W(k)$, $k$ a finite field, and $d>p$, there is an isomorphism of pro-systems
$$"\mathrm{lim}_n" H^{d}(X_1,\mathcal{K}^M_{j,X_n}/p^r)\r H^{d+j}_{\mathrm{\et}}(X_1,\mathcal{S}_r(j))$$ which relates the restriction map to the $p$-adic cycle class map. 
Here $\mathcal{S}_r(j)$ are the syntomic complexes defined in \cite{Ka87}.

\subsection*{Acknowledgement} This article is part of the author's PhD-thesis. I would like to thank my supervisor Moritz Kerz for pointing out the question of this article to me as well as a lot of help while working on it.

\section{Local cohomology and some calculations}\label{first thickening}
In this section we recall some definitions and calculations in local cohomology which we  will need later on. A standard reference for the following is \cite[Ch. IV]{Ha66}. For the convenience of the reader we recall the following definitions:
\begin{definition}
Let $X$ be a topological space and $\mathcal{F}$ a sheaf of abelian groups on $X$.
\begin{enumerate}
\item We define $\Gamma_Z(X,\mathcal{F}):=\mathrm{ker}[\Gamma(X)\r \Gamma(X-Z)].$ Let $H^i_Z(X,\mathcal{F})$ be the $i$-th right derived functor of $\Gamma_Z$. (\cite[p. 216]{Ha66}) 
\item We define $\underline{\Gamma}_Z(\mathcal{F})$ to be the sheaf whose sections on an open $U\subset X$ are given by $\Gamma_{Z\cap U}(X\cap U,\mathcal{F}|_U)$. Let $\mathcal{H}^i_Z(\mathcal{F})$ be the $i$-th right derived functor of $\underline{\Gamma}_Z$. If $i:Z\hookrightarrow X$ denotes the inclusion, we sometimes also write $R^ii^!\mathcal{F}$ for $\mathcal{H}^i_Z(\mathcal{F})$. (\cite[p. 220, Var. 3]{Ha66}) 
\item Let $x$ be a point of $X$. We define $\Gamma_x(\mathcal{F})$ to be the subgroup of $\mathcal{F}_x$ consisting of elements $\bar{s}$ which have a representative in a suitable neighborhood $U$ of $x$, whose support is $U\cap \overline{\{x\}}$. Let $H^i_x(X,\mathcal{F})$ be the $i$-th right derived functor of $\Gamma_x$. Note that $\mathcal{H}^i_{\overline{\{x\}}}(X,\mathcal{F})_x$. (\cite[p. 226, Var. 8]{Ha66}) 
\end{enumerate}
\end{definition} 

Let $X$ be a locally noetherian scheme. To any sheaf of abelian groups $\mathcal{F}$ on $X$ we can associate a coniveau complex of sheaves $$\mathcal{C}(\mathcal{F}):=\bigoplus_{x\in X^{(0)}} i_{x,*}H^{0}_x(X,\mathcal{F})\rightarrow \bigoplus_{x\in X^{(1)}} i_{x,*}H^1_x(X,\mathcal{F}))\r ...$$
where $i_{x}:x\r X$ is the natural inclusion. This complex is also called the Cousin complex of $\mathcal{F}$.
\begin{definition} A sheaf $\mathcal{F}$ on $X$ is called Cohen-Macaulay, or simply $\mathrm{CM}$, if for every $x\in X$ it holds that $H^i_x(X,\mathcal{F})=0$ for $i\neq \mathrm{codim}(x)$.
\end{definition}
Via the coniveau spectral sequence 
$$E_1^{p,q}=\bigoplus_{x\in X^{(p)}}H^{p+q}_x(X,\mathcal{F})\Rightarrow H^n(X,\mathcal{F})$$
one can easily deduce that the property of being CM for $\mathcal{F}$ is equivalent to $\mathcal{C}(\mathcal{F})$ being an acyclic resolution of $\mathcal{F}$ (see \cite[Ch. IV, Prop. 2.6]{Ha66}). In that case one can use $\mathcal{C}(\mathcal{F})$ to calculate the cohomology of $\mathcal{F}$, i.e. $H^*(X,\mathcal{F})\cong H^*(X,\mathcal{C}(\mathcal{F}))$.
Locally free sheaves are CM (see \cite[p.239]{Ha66}) so in particular the sheaf of differential forms $\Omega_{X}^{1}$ and its exterior powers $\Omega_{X}^{a}$ are CM if $X$ is a smooth variety over a field. 
\begin{lemma}\label{cal1} Let $k$ be a field and $X_1$ be a scheme of dimension $1$ over $\mathrm{Spec}(k)$. Let $x\in X_1$ be a regular closed point and $f$ a local parameter at $x$. Then 
$$\mathcal
{O}_{X_1,x}[\frac{1}{f}]/\mathcal
{O}_{X_1,x}\cong H^1_x(X_{1},\mathcal{O}_{X_1}).$$ \end{lemma}  
\begin{proof}
We calculate $H_x^1(X_1, \mathcal{O}_{X_1})$ locally as follows: Let $X_{1,x}:=\Spec(\O_{X_1,x})$. Applying Motif B of \cite[p.220]{Ha66} to the triple $(x,X_1,X_1-x)$, we get a short exact sequence $$H^0(X_{1,x}, \mathcal{O}_{X_1}|_{X_{1,x}})\rightarrow H^0(X_{1,x}-x,\mathcal{O}_{X_1}|_{X_{1,x}-x})\rightarrow H^1_x(X_{1,x},\mathcal{O}_{X_1}|_{X_{1,x}})\rightarrow H^1(X_{1,x}, \mathcal{O}_{X_1}|_{X_{1,x}}).$$
Since $H^1(X_{1,x}, \mathcal{O}_{X_1}|_{X_{1,x}})=0$, this gives an isomorphism $$\mathcal
{O}_{X_1,x}[\frac{1}{f}]/\mathcal
{O}_{X_1,x}\cong H^1_x(X_{1,x},\mathcal{O}_{X_1}|_{X_{1,x}}).$$ 
\end{proof}

We now turn to the higher dimensional case. Similar calculations can be found in \cite[Sec. 5]{Bl72}.
\begin{lemma}\label{cal2} Let $k$ be a field and $X_1$ be a separated scheme of dimension $d>1$ over $\mathrm{Spec}(k)$. Let $x\in X_1$ be a regular closed point and $f_1,...,f_d\in \mathfrak{m}_x$ a local parameter system at $x$. Then $H_x^d(X,\Omega_{X_1}^{d-1})$ is generated by elements of the form 
$$\frac{df_1\wedge...\wedge\hat{df_i}\wedge...\wedge df_d}{f_1^{n_1}...f_d^{n_d}}$$ modulo $\frac{df_1\wedge...\wedge\hat{df_i}\wedge...\wedge df_d}{f_1^{n_1}...\hat{f_j}...f_d^{n_d}}$ over $\mathcal{O}_{X_1,x}$. \end{lemma}

\begin{proof} Let $U$ be an affine neighbourhood of $x\in X_1$. Let $\mathcal{V}=\{V_i:=U-V(f_i)\}$ be a covering of $U-x$. Then the $\check{\text{C}}$ech complex 
$$0\rightarrow \prod\Omega^{d-1}_{X_1}(V_i)\rightarrow \prod_{i\neq j} \Omega^{d-1}_{X_1}(V_i\cap V_j)\rightarrow ...\rightarrow\Omega^{d-1}_{X_1}(V_1\cap...\cap V_d)$$ 
gives an isomorphism $$\text{coker}(\prod \Omega^{d-1}(V_1\cap...\cap \hat{V_i}\cap...\cap V_d)) \rightarrow\Omega^{d-1}(V_1\cap...\cap V_d))\cong \Gamma(U,R^{d-1}j_*(\Omega^{d-1}|_{U-x})),$$ where $j$ is the inclusion $X_1-x\hookrightarrow X_1$. By Motif B of \cite[p.220]{Ha66} and since $d\geq2$ there is an isomorphism $$R^{d-1}j_*(\Omega^{d-1}|_{X_1-x})\cong \mathcal{H}^d_x(X_1,\Omega^{d-1}_{X_1}).$$ In other words, $\Gamma(U,\mathcal{H}^d_x(X,\Omega^{d-1}_{X_1}))$ is generated by elements of the form $\frac{df_1\wedge...\wedge\hat{df_i}\wedge...\wedge df_d}{f_1^{n_1}...f_d^{n_d}}$ modulo $\frac{df_1\wedge...\wedge\hat{df_i}\wedge...\wedge df_d}{f_1^{n_1}...\hat{f_j}...f_d^{n_d}}$ over $\mathcal{O}(U)$. Passing to the limit, we get the desired result.
\end{proof}

In case (1) of the introduction, i.e. for $k$ a field of characteristic $0$, $S_n=\mathrm{Spec}k[t]/(t^n)$, $S=\mathrm{Spec}k[[t]]$ and $X$ smooth, separated and of finite type over $S$, there exists a short exact sequence
$$0\rightarrow \Omega^{r-1}_{X_1}\rightarrow \mathcal{K}^M_{r,X_n}\rightarrow \mathcal{K}^M_{r,X_{n-1}}\rightarrow 0$$
(see \cite[Prop. 2.3]{BEK14'}). In particular, $\mathcal{K}^M_{r,X_n}$ is CM for all $n\geq 1$ (see \cite[Prop. 3.5]{BEK14'}). We now show analogous statements for case (2).

\begin{proposition}\label{exactsequwitt}
Let $k$ be a perfect field with $\mathrm{ch}(k)=p>2$ and let $X$ be a smooth scheme over $A:=W(k)$. Then there is an exact sequence 
\begin{equation}\label{exseqBK} 0\rightarrow\Omega^{r-1}_{X_1}/B_{n-1}\Omega^{r-1}_{X_1}\rightarrow K^M_{r,X_{n+1}}\rightarrow K^M_{r,X_n}\rightarrow 0. 
\end{equation}
\end{proposition}
\begin{proof}
Let $R_n$ be an essentially smooth local ring over $A/\pi^n$. We define a filtration $U^iK^M_r(R_n)\subset K^M_r(R_n)$ by 
$$U^iK^M_r(R_n):=<\{1+\pi^ix,x_2,...,x_r\: |\: x\in R_n,x_2,...,x_r\in R_n^*\}>.$$
The $U^i$ fit into the following exact sequences:
$$0\rightarrow U^nK^M_r(R_{n+1})\rightarrow K^M_r(R_{n+1})\rightarrow K^M_r(R_n)\rightarrow 0.$$
By \cite[Proof of Prop. 12.3, Step 3]{BEK14} there is an isomorphism $$\Omega^{r-1}_{R_1}/B_{i-1}\Omega^{r-1}_{R_1}\cong gr^iK^M_r(R_n)\cong U^iK^M_r(R_n)/U^{i+1}K^M_r(R_n)$$ 
and since $U^{n+1}(K^M_r(R_{n+1}))=0$, this implies that $U^n(K^M_r(R_{n+1}))\cong \Omega^{r-1}_{R_1}/B_{n-1}\Omega^{r-1}_{R_1}$ and therefore the exact sequence 
$$0\rightarrow\Omega^{r-1}_{R_1}/B_{n-1}\Omega^{r-1}_{R_1}\rightarrow K^M_r(R_{n+1})\rightarrow K^M_r(R_n)\rightarrow 0. $$
\end{proof}

\begin{corollary}\label{KmnCM} Let $X$ be as in Proposition \ref{exactsequwitt}. Then the sheaf $\mathcal{K}^M_{r,X_n}$ is $\mathrm{CM}$.
\end{corollary}
\begin{proof} 
Applying the derived functor $H^i_x(X_1,-)$ to (\ref{exseqBK}), we get the exact sequence
$$H^i_x(X_1,\Omega^{r-1}_{R_1}/B_{n-1}\Omega^{r-1}_{R_1})\rightarrow H^i_x(X_1,\mathcal{K}^M_{r,X_2})\rightarrow H^i_x(X_1,\mathcal{K}^M_{r,X_1}).$$
By \cite[Cor. 3.9, p. 572]{Il79}, the sheaf $\Omega^{r-1}_{X_1}/B_{n-1}\Omega^{r-1}_{X_1}$ is locally free and therefore $\mathrm{CM}$. The sheaf $\mathcal{K}^M_{r,X_1}$ is $\mathrm{CM}$ by \cite{Ke09} and \cite{Ke10}.
The result follows inductively.
\end{proof}

Finally we show that the maps defined in the introduction are compatible with Gysin maps, i.e. compatible with maps induced by the closed immersion of smooth subschemes.
\begin{lemma}\label{gysinmap}
Let the notation for $X$ and $X_n$ be as in the introduction. Let $Y$ (resp. $Y_n$) be a smooth (over $A$) closed subscheme of codimension $1$ of $X$ (resp. $X_n$). Let $i:Y\r X$ (resp. $i:Y_n\r X_n$) denote the inclusion. Assume the Gersten conjecture for the Milnor K-sheaves $\mathcal{K}^M_{*,X}$ and $\mathcal{K}^M_{*,Y}$. Then there are commutative diagrams
\begin{equation}\label{diag1}
\begin{gathered}
\begin{xy}
  \xymatrix{
        H^{d}(X,\mathcal{K}^M_{d,X}) \ar[r]^{res_{X_n}}   & H^{d}(X_1,\mathcal{K}^M_{d,X_n})    \\
       H^{d-1}(Y,\mathcal{K}^M_{d-1,Y})   \ar[r]^{res_{Y_n}}  \ar[u]  &   H^{d-1}(Y_1,\mathcal{K}^M_{d-1,Y_n}) \ar[u] 
  }
\end{xy} 
\end{gathered}
\end{equation}
and
\begin{equation}\label{diag2}
\begin{gathered}\begin{xy}
  \xymatrix{
       H^d(X_1,\Omega^{d-1}_{X_1})  \ar[r]  & H^{d}(X_1,\mathcal{K}^M_{d,X_n})  \\
      H^{d-1}(Y_1,\Omega^{d-2}_{Y_1})  \ar[r]  \ar[u]  & H^{d-1}(Y_1,\mathcal{K}^M_{d-1,Y_n})    \ar[u] 
  }
\end{xy} 
\end{gathered}
\end{equation}
in which the vertical maps are Gysin maps which are defined in the course of the proof.
\end{lemma}
\begin{proof}
Let $f_d$ be a local parameter defining $Y_1$. 

We start with diagram (\ref{diag1}). There is a commutative diagram of sheaves
$$\begin{xy}
  \xymatrix{
       \mathcal{K}^M_{d-1,Y} \ar[r] \ar[d] & R^1i^!\mathcal{K}^M_{d,X} \ar[d] \\
      \mathcal{K}^M_{d-1,Y_n}  \ar[r]    & R^1i^!\mathcal{K}^M_{d,X_n}     
  }
\end{xy} $$
in which the horizontal maps are given by $\{y_1,...,y_{d-1}\}\mapsto \{y_1,...,y_{d-1}, f_d\}$. The vertical maps are induced by reduction mod $\pi^n$. By functoriality this implies that there is a commutative diagram 
$$\begin{xy}
  \xymatrix{
       H^{d-1}(Y,\mathcal{K}^M_{d-1,Y}) \ar[d] \ar[r] &     H^{d-1}(Y,R^1i^!\mathcal{K}^M_{d,X}) \ar[d] \ar[r]^-{\cong}  & H^d_{Y}(X,\mathcal{K}^M_{d,X}) \ar[d] \ar[r] & H^d(X,\mathcal{K}^M_{d,X}) \ar[d] \\
 H^{d-1}(Y_1,\mathcal{K}^M_{d-1,Y_n}) \ar[r] & H^{d-1}(Y_1,R^1i^!\mathcal{K}^M_{d,X_n})       \ar[r]^-{\cong}     & H^{d}_{Y_1}(X_1,\mathcal{K}^M_{d,X_n})     \ar[r] & H^{d}(X_1,\mathcal{K}^M_{d,X_n}). 
  }
\end{xy} $$
The isomorphisms follow from the spectral sequences
$$H^i(Y,R^ji^!\mathcal{K}^M_{d,X})\Rightarrow H^{i+j}_Y(X,\mathcal{K}^M_{d,X})$$
and
$$H^i(Y_1,R^ji^!\mathcal{K}^M_{d,X_n})\Rightarrow H^{i+j}_{Y_1}(X_1,\mathcal{K}^M_{d,X_n})$$
and the fact that $\mathcal{K}^M_{d,X}$ and $\mathcal{K}^M_{d,X_n}$ are CM. In the first case this follows from the assumption of the Gersten conjecture. In the second case this follows from \cite[Prop. 3.5]{BEK14'} and Corollary \ref{KmnCM}.

For (\ref{diag2}) notice that there is a commutative diagram of sheaves
$$\begin{xy}
  \xymatrix{
       \Omega^{d-2}_{Y_1}  \ar[r] \ar[d] & R^1i^!\Omega^{d-1}_{X_1} \ar[d] \\
      \mathcal{K}^M_{d-1,Y_n}  \ar[r]    & R^1i^!\mathcal{K}^M_{d,X_n}     
  }
\end{xy} $$
which on elements is given by 
$$\begin{xy}
  \xymatrix{
      xd\text{log}(y_1)\wedge ...\wedge d\text{log}(y_{d-2})  \ar@{|->}[r] \ar@{|->}[d] & xd\text{log}(y_1)\wedge ...\wedge d\text{log}(y_{d-2}) \wedge\mathrm{dlog}f_d  \ar@{|->}[d] \\
      \{1+x\pi^{n-1},y_1,...,y_{d-2}\}  \ar@{|->}[r]    & \{1+x\pi^{n-1},y_1,...,y_{d-2},f_d\}.
  }
\end{xy} $$
By functoriality this implies that there is a commutative diagram 
$$\begin{xy}
  \xymatrix{
       H^{d-1}(Y_1,\Omega^{d-2}_{Y_1}) \ar[d] \ar[r] &     H^{d-1}(Y_1,R^1i^!\Omega^{d-1}_{X_1}) \ar[d] \ar[r]^-{\cong}  & H^d_{Y_1}(X_1,\Omega^{d-1}_{X_1}) \ar[d] \ar[r] & H^d(X_1,\Omega^{d-1}_{X_1}) \ar[d] \\
 H^{d-1}(Y_1,\mathcal{K}^M_{d-1,Y_n}) \ar[r] & H^{d-1}(Y_1,R^1i^!\mathcal{K}^M_{d,X_n})       \ar[r]^-{\cong}     & H^{d}_{Y_1}(X_1,\mathcal{K}^M_{d,X_n})     \ar[r] & H^{d}(X_1,\mathcal{K}^M_{d,X_n}). 
  }
\end{xy} $$
The isomorphisms follow from the same spectral sequence argument as above and the fact that $\Omega^{d-1}_{X_1}$ and $\mathcal{K}^M_{d,X_n}$ are CM.
\end{proof}

\section{Pro-objects and the relative dimension $1$ case}\label{reldim1}

In this section, we quickly review the theory of pro-objects. Standard references are \cite{AM69} and \cite{SGA4}. We then prove in Theorem \ref{conjrel1} that the conjecture of Kerz, Esnault and Wittenberg mentioned in the introduction holds in the relative dimension $1$ case.

Let $\mathcal{C}$ be a category. The category of pro-objects pro-$\mathcal{C}$ in $\mathcal{C}$ is defined as follows:
A pro-object is a contravariant functor $$X: I^{\circ}\rightarrow \mathcal{C},$$ 
from a filtered index category $I$ to $\mathcal{C}$, i.e. an inverse system of objects $X_i$ in $\mathcal{C}$. We denote $X$ also by $"\text{lim}" X_i$ or $(X_i)_i$. The morphisms between two objects $X="\text{lim}" X_i$ and $Y="\text{lim}" Y_i\in$ in pro-$\mathcal{C}$ are given by 
$$\mathrm{Hom}(X,Y)=\varprojlim_j(\varinjlim_i \mathrm{Hom}(X_i,Y_j)).$$ 

There is a natural fully faithful embedding of $\mathcal{C}$ into pro-$\mathcal{C}$ which associates to an object $C\in \mathcal{C}$ the constant diagram $C$. This functor has a right adjoint pro-$\mathcal{C}\rightarrow \mathcal{C}, "\text{lim}"X_i\mapsto \varprojlim_iX_i$. If $\mathcal{C}$ has finite  direct (inverse) limits, then the functor 
$$\mathrm{Hom}(I^{\circ},\mathcal{C})\rightarrow \text{pro-}\mathcal{C}$$
commutes with finite direct (inverse) limits. In particular if $\mathcal{C}$ has finite direct and inverse limits, then the above functor is exact (see \cite[p.163]{AM69}).
 
A criterion for when a map of pro-systems is an isomorphism is given by the following proposition (see \cite[Lem. 2.3]{Is01}):

\begin{proposition}\label{isomorphismprosystems}
A level map $A\r B$ in $\mathrm{pro}$-$\mathcal{C}$, i.e. a map between pro-systems with the same index category and maps $A_s\r B_s$ for all $s\in I$, is an isomorphism if and only if for all $s$ there exists a $t\geq s$ and a commutative diagram
$$\begin{xy}
  \xymatrix{
        A_t \ar[d]_{} \ar[r]^{} &  B_t \ar[d] \ar[dl] \\
     A_s \ar[r] & B_s .
  }
\end{xy} $$
\end{proposition}

\begin{theorem}\label{conjrel1} Let $k$ be a finite field of characteristic $p>2$ and $A=W(k)$ the Witt ring of $k$. Let $X$ be a smooth projective scheme of relative dimension $1$ over $A$. Then the map $$res:\CH^1(X)/p^i\rightarrow "\mathrm{lim}" H^1(X_1,\mathcal{K}^M_{1,X_n}/p^i)$$
is an isomorphism in the category of pro-systems of abelian groups. 
\end{theorem}
\begin{proof}
We first note that $\CH^1(X)\cong \Pic(X)$ and that $\Pic(X)\cong \varprojlim \Pic(X_n)$ by \cite[Thm. 5.1.4]{EGA3}. Furthermore, $H^1(X_1,\mathcal{K}^M_{1,X_n})=H^1(X_1,\mathcal{O}_{X_n}^\times)\cong \Pic(X_n)$. 
It therefore suffices to show that
$$\varprojlim \Pic(X_n)\otimes\mathbb{Z}/p^i\mathbb{Z}\rightarrow "\text{lim}" \Pic(X_n)\otimes\mathbb{Z}/p^i\mathbb{Z}$$
is an isomorphism.
 
Using the $p$-adic logarithm isomorphism $1+p\mathcal{O}_{X_n}\xrightarrow{\cong} p\mathcal{O}_{X_n}$, the short exact sequence 
$$1\rightarrow (1+p^j\mathcal{O}_{X_n})\rightarrow \mathcal{O}^{\times}_{X_n}\rightarrow \mathcal{O}^{\times}_{X_j}\rightarrow 1$$
induces a short exact sequence 
$$0\rightarrow H^1(X_1,p^j\mathcal{O}_{X_n})\rightarrow H^1(X_1,\mathcal{O}_{X_n}^*)\rightarrow H^1(X_1,\mathcal{O}_{X_j}^*) \rightarrow H^2(X_1,p^j\mathcal{O}_{X_n})=0 
$$ (the last equality following for dimension reasons). Applying the Functor $\varprojlim_n$, we get an exact sequence 
$$\varprojlim_n H^1(X_1,p^j\mathcal{O}_{X_n})\rightarrow \varprojlim_n \Pic(X_n)\rightarrow \Pic(X_j)\rightarrow \varprojlim_n {^1}H^1(X_1,p^j\mathcal{O}_{X_n}).$$
Now $\varprojlim_n^1 H^1(X_1,p^j\mathcal{O}_{X_n})=0$ since the inverse system $(H^1(X_1,p^j\mathcal{O}_{X_n}))_n$ satisfies Mittag-Leffler being an inverse system of finite dimensional vector spaces. Tensoring with $\mathbb{Z}/p^i\mathbb{Z}$ gives the exact sequence
$$\varprojlim_n H^1(X_1,p^j\mathcal{O}_{X_n})\otimes\mathbb{Z}/p^i\mathbb{Z}\rightarrow \varprojlim_n \Pic(X_n)\otimes\mathbb{Z}/p^i\mathbb{Z}\rightarrow \Pic(X_j)\otimes\mathbb{Z}/p^i\mathbb{Z}\rightarrow 0.$$
We now apply the exact functor $"\varprojlim_j"$ to this sequence. By the theorem on formal functions, there is an isomorphism $$"\varprojlim_j"\varprojlim_n H^1(X_1,p^j\mathcal{O}_{X_n})\otimes\mathbb{Z}/p^i\mathbb{Z} \cong "\varprojlim_j" H^1(X,p^j\mathcal{O}_{X})\otimes\mathbb{Z}/p^i\mathbb{Z}.$$ Since the image of the inclusion $p^{i+j}\mathcal{O}_X\hookrightarrow p^j\mathcal{O}_X$ vanishes modulo $p^i$, the same holds for the image of the morphism $H^1(X,p^{i+j}\mathcal{O}_X)\rightarrow H^1(X,p^j\mathcal{O}_X)$.  By Proposition \ref{isomorphismprosystems} this implies that $$"\varprojlim_j"\varprojlim_n H^1(X_1,p^j\mathcal{O}_{X_n})\otimes\mathbb{Z}/p^i\mathbb{Z}$$ is pro-isomorphic to zero and therefore that the theorem holds.
\end{proof}

\section{Lefschetz theorems over arbitrary fields}\label{lefschetzsection}
In this section we prove a Kodaira vanishing theorem which implies a Lefschetz theorem allowing us later in Section \ref{secsurj} to reduce our main theorem to relative dimension $2$. The techniques and statements we develop work over arbitrary fields and might be well known to the expert.  

To put the following proposition into context, we recall the Kodaira vanishing theorem. A good reference is \cite{EV92}.

\begin{theorem}
Let $X$ be a complex projective manifold and $\mathcal{A}$ an ample invertible sheaf. Then
$$H^a(X,\Omega^b_X\otimes \mathcal{A})=0$$
for $a+b> \mathrm{dim} X$.
\end{theorem}

In this section let $X_1$, unless otherwise stated, be a smooth projective scheme over a field $k$. Let $H\subset X_1$ be a hyperplane section and $\mathcal{L}(d)=|dH|, d>0,$ be the linear system of hypersurface sections of degree $d$. We say that a hypersurface section $Y_1\subset X_1$ is of high or sufficiently high degree if $Y_1\in \mathcal{L}(d)$ with $d$ sufficiently large such that certain higher cohomology groups vanish by Serre vanishing. Note furthermore that if $A$ is a discrete valuation ring with residue field $k$, then a hypersurface section of high degree of $\mathbb{P}^N_k$ may be lifted to $\mathbb{P}^N_A$ since there is an exact sequence
$$\O_{\mathbb{P}^N_A}(d)(\mathbb{P}^N_A)\r \O_{\mathbb{P}^N_k}(d)(\mathbb{P}^N_k)\r H^1(\mathbb{P}^N_A,\O_{\mathbb{P}^N_A}(-\mathbb{P}^N_k)(d))$$
and $H^1(\mathbb{P}^N_A,\O_{\mathbb{P}^N_A}(-\mathbb{P}^N_k)(d))=0$ for $d$ large by Serre vanishing.

\begin{proposition}\label{kodaira} Let $Y_1$ be a smooth hypersurface of $X_1$ and $d=\mathrm{dim} X_1$. If $Y_1$ is of sufficiently high degree, then
$$H^{a}(X_1,\Omega^{b}_{Y_1}\otimes_{\mathcal{O}_{X_1}}\mathcal{O}_{X_1}(Y_1))=H^{a}(Y_1,\Omega^{b}_{Y_1}\otimes_{\mathcal{O}_{Y_1}}\mathcal{O}_{X_1}(Y_1)|_{\mathcal{O}_{Y_1}})=0$$
for $a+b>d-1$.
\end{proposition}
\begin{proof}
Note that the first equality in the statement follows from the projection formula $i_*\Omega^{b}_{Y_1}\otimes_{\mathcal{O}_{X_1}}\mathcal{O}_{X_1}(Y_1)=i_*(\Omega^{b}_{Y_1}\otimes_{\mathcal{O}_{Y_1}}i^*\mathcal{O}_{X_1}(Y_1))$ for $i$ the inclusion $Y_1\hookrightarrow X_1$. 

We first show that for $\omega_{X_1}=\Omega^d_{X_1}$ and $\omega_{Y_1}=\Omega^{d-1}_{Y_1}$, we have that 
$$H^{a>0}(X_1,\omega_{Y_1}\otimes_{\mathcal{O}_{X_1}}\mathcal{O}_{X_1}(Y_1))=0$$ 
if $Y_1$ is of high degree. By \cite[Ch. II, Prop. 8.20]{Ha77} we know that 
$$\omega_{Y_1}\cong \omega_{X_1}\otimes \O_{Y_1} \otimes \O_{X_1}(Y_1).$$
This implies that $\omega_{X_1}|_{Y_1}=\omega_{Y_1}(-Y_1)$ and therefore that the sequences 
$$0\rightarrow \omega_{X_1}(Y_1)\rightarrow \omega_{X_1}(2Y_1)\rightarrow \omega_{Y_1}(Y_1) \rightarrow 0$$
and
$$H^a(X_1,\omega_{X_1}(2Y_1))\r H^a(X_1,\omega_{Y_1}(Y_1))\r H^{a+1}(X_1,\omega_{X_1}(Y_1))$$
are exact.
Since by Serre vanishing $H^a(X_1,\omega_{X_1}(2Y_1))=H^a(X_1,\omega_{X_1}(Y_1))=0$ for $a>0$ and $Y_1$ of sufficiently high degree, this implies that if $Y_1$ is of sufficiently high degree we also have that $H^{a>0}(X_1,\omega_{Y_1}(Y_1))=H^{a>0}(Y_1,\omega_{Y_1}\otimes_{\mathcal{O}_{X_1}}\mathcal{O}_{X_1}(Y_1)|_{\O_{Y_1}})=0$.

We now consider the exact sequence
$$0\rightarrow \Omega^{p-1}_{Y_1}(-Y_1) \rightarrow \Omega^{p}_{X_1}|_{Y_1} \rightarrow \Omega^{p}_{Y_1} \rightarrow 0$$
coming from the conormal exact sequence $0\rightarrow \O_{Y_1}(-Y_1) \rightarrow \Omega^{1}_{X_1}|_{Y_1} \rightarrow \Omega^{1}_{Y_1} \rightarrow 0$. Tensoring with $\O_{Y_1}(2Y_1)$ gives an exact sequence 
$$0\rightarrow \Omega^{p-1}_{Y_1}(Y_1) \rightarrow \Omega^{p}_{X_1}(2Y_1)|_{Y_1} \rightarrow \Omega^{p}_{Y_1}(2Y_1) \rightarrow 0.$$
This implies that the sequence
$$H^{a}(X_1,\Omega_{Y_1}^b(2Y_1))\r H^{a+1}(X_1,\Omega^{b-1}_{Y_1}(Y_1))\r H^{a+1}(X_1,\Omega^b_{X_1}(2Y_1))$$
is exact. The proposition follows inductively.
\end{proof}

\begin{remark}\label{remarkDI}
The referee made the author aware of the fact that Proposition \ref{kodaira} is closely related to the following Kodaira type vanishing theorem of Deligne-Illusie (see \cite[p. 248]{DI}): let $X$ be a smooth projective and equidimensional scheme over a field $k$ of characteristic $p$ and assume that $X$ is liftable to $W_2(k)$. Let $L$ be an ample invertible sheaf on $X$. Then
$$H^j(X,\Omega_X^j\otimes L^{-1})=0$$
for $i+j<\mathrm{inf}(p,\mathrm{dim}X=d)$ and in particular
$$H^j(X,\Omega_X^j\otimes L)=0$$
for $i+j>d$ and $p\geq d$. Since $Y_1$ is liftable to $W(k)$ if it is of high degree, the cited theorem implies that $H^{a}(Y_1,\Omega^{b}_{Y_1}\otimes_{\mathcal{O}_{Y_1}}\mathcal{O}_{X_1}(Y_1)|_{\mathcal{O}_{Y_1}})=0$ if $p\geq d-1$ and $a+b>d-1$. For this note that $\mathcal{O}_{X_1}(Y_1)|_{\mathcal{O}_{Y_1}}$ is ample.
\end{remark}

We can now deduce the following Lefschetz theorem:

\begin{proposition}\label{lefschetz} Let $Y_1$ be a smooth hypersurface section of $X_1$ and $d=\mathrm{dim} X_1$. Then there is a map 
$$gys:H^{d-1}(Y_1,\Omega^{d-2}_{Y_1})\rightarrow H^d(X_1,\Omega^{d-1}_{X_1})$$ 
which is an isomorphism for $d\geq 4$ and surjective for $d=3$ if $Y_1$ is of high degree.
\end{proposition} 
\begin{proof} Let $i$ denote the inclusion $Y_1\hookrightarrow X_1$. As in Lemma \ref{gysinmap} we define $gys$ to be the composition 
$$H^{d-1}(Y_1,\Omega^{d-2}_{Y_1})\xrightarrow{} H^{d-1}(Y_1,R^1i^{!}\Omega^{d-1}_{X_1})\cong H^{d}_{Y_1}(X_1,\Omega^{d-1}_{X_1})\r H^d(X_1,\Omega^{d-1}_{X_1})$$
where the first map is induced by the Gysin map
$$g:\Omega^{d-2}_{Y_1}\rightarrow R^1i^!\Omega^{d-1}_{X_1}, \omega\mapsto \omega \wedge\frac{df_d}{f_d}$$
(see \cite[Ch. II, (3.2.13)]{Gr85}) with $f_d$ is the regular parameter defining $Y_1$.
Since $H^d(X_1-Y_1,\Omega^{d-1}_{X_1-Y_1})=0$ for $d\geq 1$, we have that 
$$H^{d}_{Y_1}(X_1,\Omega^{d-1}_{X_1})\cong H^d(X_1,\Omega^{d-1}_{X_1})$$
for $d\geq 2$. We are therefore reduced to showing that $g$ induces an isomorphism on $H^{d-1}$ for $d-1\geq 3$ and a surjection for $d-1=2$. We define a filtration 
$$g(\Omega^{d-2}_{Y_1})=\mathcal{F}_1\subset \mathcal{F}_2\subset...\subset \cup_{i\geq 0}\mathcal{F}= R^1i^{!}\Omega^{d-1}_{X_1},$$ 
letting $\mathcal{F}_i$ be the subsheaf of $R^1i^{!}\Omega^{d-1}_{X_1}$ locally defined by
$$<\omega\wedge \frac{df_d}{f_d^{n_d}}|n_d\geq i>.$$
Here $\omega\in\Omega^{d-2}_{Y_1}$.  

Let $gr_i R^1i^{!}\Omega^{d-1}_{X_1}:= \mathcal{F}_{i+1}/\mathcal{F}_i$. Then $gr_i R^1i^{!}\Omega^{d-1}_{X_1}\cong \Omega^{d-2}_{Y_1}\otimes_{\mathcal{O}_{Y_1}}\mathcal{O}_{X_1}(Y_1)|_{\mathcal{O}_{Y_1}}$ and the short exact sequence 
$$0\rightarrow \mathcal{F}_i\rightarrow \mathcal{F}_{i+1}\rightarrow gr_i R^1i^{!}\Omega^{d-1}_{X_1}\rightarrow 0$$
induces the following exact sequence on cohomology groups:
\begin{equation}
\begin{split} H^{d-2}(Y_1,\Omega^{d-2}_{Y_1}\otimes_{\mathcal{O}_{Y_1}}\mathcal{O}_{X_1}(Y_1)|_{\mathcal{O}_{Y_1}})\rightarrow H^{d-1}(Y_1,\mathcal{F}_i)\rightarrow H^{d-1}(Y_1,\mathcal{F}_{i+1}) \\ 
\rightarrow H^{d-1}(Y_1,\Omega^{d-2}_{Y_1}\otimes_{\mathcal{O}_{Y_1}}\mathcal{O}_{X_1}(Y_1)|_{\mathcal{O}_{Y_1}})
\end{split}
\end{equation}
By Proposition \ref{kodaira} we have that if $Y_1$ is of high degree, then $H^{d-1}(Y_1,\Omega^{d-2}_{Y_1}\otimes_{\mathcal{O}_{X_1}}\mathcal{O}_{X_1}(Y_1)|_{\mathcal{O}_{Y_1}})$ vanishes for $d>2$ and $H^{d-2}(Y_1,\Omega^{d-2}_{Y_1}\otimes_{\mathcal{O}_{X_1}}\mathcal{O}_{X_1}(Y_1)|_{\mathcal{O}_{Y_1}})$ for $d>3$.

This implies that the maps $H^{a}(Y_1,\mathcal{F}_i)\rightarrow H^{a}(Y_1,\mathcal{F}_{i+1})$ are isomorphisms for $d\geq 4$ and surjective for $d=3$ if $Y_1$ is of sufficiently high degree. Since $H^a(Y_1,\varinjlim\mathcal{F}_i)\cong \varinjlim H^a(Y_1,\mathcal{F}_i)$ (see \cite[Ch. III, Prop. 2.9]{Ha77}), the same holds for the maps $H^{a}(Y_1,\mathcal{F}_1=\Omega^{d-2}_{Y_1})\rightarrow H^{a}(Y_1, R^1i^{!}\Omega^{d-1}_{X_1})$. 
In particular, for $d=\text{dim}X_1=3$ we get that $H^{d-1}(Y_1,\Omega^{d-2}_{Y_1})\rightarrow H^d(X_1,\Omega^{d-1}_{X_1})$ is surjective and for $d=\text{dim}X_1\geq 4$ that $H^{d-1}(Y_1,\Omega^{d-2}_{Y_1})\rightarrow H^d(X_1,\Omega^{d-1}_{X_1})$ is an isomorphism.
\end{proof}

\begin{corollary}\label{cord} Let $X$ be as in the introduction and $d\geq 3$. Let $Y_1$ a smooth hypersurface section of $X_1$. Let $\alpha\in H^d(X_1,\mathcal{K}^M_{d,X_n})$. If $Y_1$ is of sufficiently high degree and contains the image of $\alpha$ in $\CH_0(X_1)$ under the restriction map $H^d(X_1,\mathcal{K}^M_{d,X_n})\rightarrow H^{d}(X_1,\mathcal{K}^M_{d,X_1})\cong \CH_0(X_1)$, then $Y_1$ lifts to a smooth projective subscheme $Y$ of $X$ over $A$ and $\alpha$ is in the image of $H^{d-1}(Y_1,\mathcal{K}^M_{d-1,Y_n})\rightarrow H^d(X_1,\mathcal{K}^M_{d,X_n})$.
\end{corollary}
\begin{proof}
That $Y_1$ lifts to a smooth projective subscheme $Y$ of $X$ over $A$ if it is of high degree follows from Serre vanishing. We do the $n=2$ case. The general case follows inductively. By Lemma \ref{gysinmap} there is a commutative diagram
$$\begin{xy}
  \xymatrix{
       H^d(X_1,\Omega^{d-1}_{X_1})  \ar[r]  & H^{d}(X_1,\mathcal{K}^M_{d,X_2}) \ar[r]   & H^{d}(X_1,\mathcal{K}^M_{d,X_1}) \ar[r] & 0  \\
      H^{d-1}(Y_1,\Omega^{d-2}_{Y_1})  \ar[r]  \ar[u]  & H^{d-1}(Y_1,\mathcal{K}^M_{d-1,Y_2})   \ar[r]^{} \ar[u]  &   H^{d-1}(Y_1,\mathcal{K}^M_{d-1,Y_1}) \ar[u] \ar[r] & 0
  }
\end{xy} $$
in which the rows are induced by the (right-)exact sequence of sheaves 
$$\Omega^{d-1}\rightarrow \mathcal{K}^M_{2}\rightarrow \mathcal{K}^M_{1}\rightarrow 0$$
on $X_1$ and $Y_1$. The statement now follows from Proposition \ref{lefschetz} and a simple diagram chase. 
\end{proof}

For the sake of completeness we also prove the following Lefschetz theorem. We will not use it though.

\begin{proposition}\label{lefschetzinjective} Let $Y_1$ be a smooth hypersurface section of $X_1$ and $d=\mathrm{dim} X_1$. Let $i$ denote the inclusion $Y_1\hookrightarrow X_1$. Then the map $$i^*:H^{q}(X_1,\Omega^{p}_{X_1})\rightarrow H^q(Y_1,\Omega^{p}_{Y_1})$$ is an isomorphism for $p+q<d-1$ and injective for $p+q=d-1$ if $Y_1$ is of high degree.
\end{proposition} 
\begin{proof}
We factorise the map $i^*:\Omega^{p}_{X_1}\r i_*\Omega^{p}_{Y_1}$ by
$$\Omega^{p}_{X_1}\r i_*(\Omega^{p}_{X_1}|_{Y_1})$$
followed by
$$i_*(\Omega^{p}_{X_1}|_{Y_1})\r i_*\Omega^{p}_{Y_1}$$
and show that each of these maps induce isomorphisms, resp. injections, on cohomology in the stated range.

We first consider the exact sequence 
$$0\rightarrow \Omega^{p}_{X_1}(-Y_1) \rightarrow \Omega^{p}_{X_1} \rightarrow \Omega^{p}_{X_1}|_{Y_1} \rightarrow 0.$$
This induces the exact sequence 
$$H^q(X_1,\Omega^{p}_{X_1}(-Y_1)) \rightarrow H^q(X_1,\Omega^{p}_{X_1}) \rightarrow H^q(Y_1,\Omega^{p}_{X_1}|_{Y_1}) \r H^{q+1}(X_1,\Omega^{p}_{X_1}(-Y_1)). $$
By Serre duality $H^{q}(X,\Omega^{p}_{X_1}(-Y_1))\cong H^{d-q}(X,\Omega^{d-p}_{X_1}(Y_1))$. This implies that $H^q(X_1,\Omega^{p}_{X_1}) \rightarrow H^q(Y_1,\Omega^{p}_{X_1}|_{Y_1})$ is an isomorphism for $p+q<d-1$ and injective for $p+q=d-1$ if $Y_1$ is of sufficiently high degree by Serre vanishing.

We now consider the exact sequence
$$0\rightarrow \Omega^{p-1}_{Y_1}(-Y_1) \rightarrow \Omega^{p}_{X_1}|_{Y_1} \rightarrow \Omega^{p}_{Y_1} \rightarrow 0$$
on $Y_1$.
This induces the exact sequence 
$$H^q(Y_1,\Omega^{p-1}_{Y_1}(-Y_1)) \rightarrow H^q(Y_1,\Omega^{p}_{X_1}|_{Y_1}) \rightarrow H^q(Y_1,\Omega^{p}_{Y_1}) \r H^{q+1}(Y_1,\Omega^{p-1}_{Y_1}(-Y_1)) $$
which by Serre duality and Proposition \ref{kodaira} implies that $H^q(Y_1,\Omega^{p}_{X_1}|_{Y_1}) \rightarrow H^q(Y,\Omega^{p}_{Y_1})$ is an isomorphism for $p+q<d+1$ and injective for $p+q=d+2$ if $Y_1$.
\end{proof}

\begin{remark}
Proposition \ref{lefschetzinjective} is an analogue of the following theorem (see for example \cite[Thm. 1.29]{Voi07}): Let $X$ be an $n$-dimensional compact complex variety and let $Y\hookrightarrow X$ be a smooth hypersurface such that the line bundle $\O_X(Y)=(\mathcal{I}_Y)^*$ is ample. Then the restriction $$j^*:H^k(X,\Q)\r H^k(Y,\Q)$$ is an isomorphism for $k<n-1$ and injective for $k=n-1$.

This theorem may be deduced from Kodaira vanishing using the Hodge decomposition for $Y_{\C}$ and $X_{\C}$. We note that when working over an arbitrary field $k$, one seems to need the extra assumption that $Y$ is of high degree to obtain Lefschetz theorems.
\end{remark}

\section{Main theorem}\label{secsurj}
We return to the situation of the introduction. Let $A$ be a henselian discrete valuation ring with uniformising parameter $\pi$ and residue field $k$. Let $X$ be a smooth projective scheme over Spec$(A)$ of relative dimension $d$. Let $X_n:=X\times_A A/(\pi^n)$, i.e. $X_1$ is the special fiber and the $X_n$ are the respective thickenings of $X_1$. We assume furthermore that either (1) $k$ is of characteristic $0$ and $A=k[[\pi]]$ or that (2) $A$ is the Witt ring $W(k)$ of a perfect field $k$ of ch$(k)>2$.

Let us first recall how one can lift a regular closed point $x\in X_1$ to a $1$-cycle on $X$: Let $\{f_1,...,f_d\}\subset \mathcal{O}_{X_1,x}$ be a generating set of local parameters and let $\{\tilde{f_1},...,\tilde{f_d}\}$ be lifts of these generators to $\mathcal{O}_{X,x}$. The ideal $\tilde{f_1}\mathcal{O}_{X,x}+...+\tilde{f_d}\mathcal{O}_{X,x}$ defines a subscheme of Spec$\mathcal{O}_{X,x}$ and its closure in $X$ defines a subscheme $Z$ of $X$. The unique irreducible component of $Z$ containing $x$ is a prime-cycle $C\in Z_1(X)$ which is flat and finite over $A$. Such liftings are of course not unique.

The main theme of this section is the lifting of differential forms to one-cycles on $X$. Let us consider the case where $X$ is of relative dimension $1$ over $A$. Let $F_n$ be the subgroup of $\CH_1(X)$ generated by all cycles $Z$ vanishing on $X_n$, i.e. $Z|_{X_n}=0$. By Lemma \ref{cal1} we know that $H^1_x(X_1,\mathcal{O}_{X_1})\cong {O}_{X_1,x}[\frac{1}{f}]/\mathcal{O}_{X_1,x}$ for a local parameter $f\in\mathcal{O}_{X_1,x}$. In this case we can define a map 
$$\gamma_x:{O}_{X_1,x}[\frac{1}{f}]/\mathcal{O}_{X_1,x}\rightarrow \CH_1(X)/F_n$$
by 
$$\alpha=\frac{\alpha_0}{f^m}\mapsto V(\tilde{f}^m+\alpha_0\pi^{n-1})-V(\tilde{f}^m)$$
where $\tilde{f}\in \mathcal{O}_{X,x}$ is a lifting of $f$.
That $\gamma_x$ is well-defined can be seen as follows: Let $\tilde{f_1},\tilde{f_2}\in \mathcal{O}_{X,x}$ be liftings of $f$. Then $(1+\frac{\alpha_0}{\tilde{f_1^m}}\pi^{n-1})/(1+\frac{\alpha_0}{\tilde{f_2^m}}\pi^{n-1})\equiv (1+(\frac{\alpha_0}{\tilde{f_1^m}}-\frac{\alpha_0}{\tilde{f_2^m}})\pi^{n-1})=1$ mod $\pi_n$. 


In this section we show the following proposition: 

\begin{proposition}\label{surj2} Let $X$ be of relative dimension $2$ over $A$. Then, assuming the Gersten conjecture for the Milnor K-sheaf $\mathcal{K}^M_{*,X}$, the map $$res_{X_n}:\CH_1(X)\rightarrow H^2(X_1,\mathcal{K}^M_{2,X_n})$$ is surjective. In particular the map 
$res:\CH_1(X)\rightarrow "\mathrm{lim}_n" H^2(X_1,\mathcal{K}^M_{2,X_n})$
is an epimorphism in $\text{pro-}\text{Ab}$.
\end{proposition}

We need some preparation for the proof. From now on we assume that $d=2$ and in particular that dim$X_1=2$. Consider the (right-)exact sequence
$$\Omega^{1}_{X_1}\rightarrow \mathcal{K}^M_{2,X_2}\rightarrow \mathcal{K}^M_{2,X_1}\rightarrow 0.$$
We will lift elements which lie in the kernel of $res:H^2(X_1,\mathcal{K}^M_{2,X_2})\rightarrow  H^2(X_1,\mathcal{K}^M_{2,X_1})$ in a compatible way to $\CH_1(X)$. The kernel of $res$ is in the image of $H^2(X_1,\Omega^1_{X_1})$. Now since $\Omega^{1}_{X_1}$ is CM, $H^2(X_1,\Omega^1_{X_1})$ is isomorphic to  $$\text{coker}(\oplus_{x\in X^{(1)}_1} H^{1}_x(X_1,\Omega^{1}_{X_1})\rightarrow \oplus_{x\in X^{(2)}_1} H^2_x(X_1,\Omega^{1}_{X_1})).$$

In order to proceed, we need to study this cokernel and the occurring local cohomology groups a bit further. By Lemma \ref{cal2} we have that $H^2_x(X_1,\Omega^{1}_{X_1})$ is generated by differential forms of the form
$$\frac{df_1}{f_1^{n_1}f_2^{n_2}}\mathcal{O}_{X_1,x}\oplus \frac{df_2}{f_1^{n_1'}f_2^{n_2'}}\mathcal{O}_{X_1,x} \quad \text{mod} \quad \frac{df_i}{f_j^{n_j}}\mathcal{O}_{X_1,x}$$
for $\{f_1,f_2\}$ a system of local parameters in $\mathcal{O}_{X_1,x}$ and $i,j\in\{1,2\}$. 

We define subgroups $$F_r:= <\frac{df_1}{f_1^{n_1}f_2^{n_2}}+\frac{df_2}{f_1^{n_1'}f_2^{n_2'}}|n_1+n_2-1\leq r,n_1'+n_2'-1\leq r>$$
of $H^2_x(X_1,\Omega^1_{X_1})$ with respect to a system of local parameters. Sometimes we therefore write $H^2_x(X_1,\Omega^1_{X_1})_{(f_1,f_2)}$ instead of $H^2_x(X_1,\Omega^1_{X_1})$ to indicate with respect to which system of local parameters we are working. Then 
$$0\subset F_1\subset F_2\subset...\subset H^2_x(X_1,\Omega^1_{X_1})_{(f_1,f_2)}$$
defines an ascending filtration on $H^2_x(X_1,\Omega^1_{X_1})$.
We will call elements of $F_1$ forms with simple poles. The following lemma shows that this definition is in fact independent of the chosen parameter system, meaning that there is a natural isomorphism between $H^2_x(X_1,\Omega^1_{X_1})_{(f_1,f_2)}$ and $H^2_x(X_1,\Omega^1_{X_1})_{(f_1',f_2')}$ for two local parameter systems $\{f_1,f_2\}$ and $\{f_1',f_2'\}$ inducing isomorphisms on the respective filtrations. This isomorphism is given by considering a differential form with respect to the respective defining parameter systems.
\begin{lemma}\label{invariance} \begin{enumerate}
\item[(1)] Let  $x\in X_1$ be a closed point. Then subgroups $F_r\subset H^2_x(X_1,\Omega^1_{X_1})$ are independent of the local parameter system we consider them in. 
\item[(2)] In particular, the subgroup $$F_1=<\frac{df_1}{f_1f_2}\mathcal{O}_{X_1,x}\oplus \frac{df_2}{f_1f_2}\mathcal{O}_{X_1,x}>\subset H^2_x(X_1,\Omega^1_{X_1})$$
is independent of the chosen local parameter system of $\mathcal{O}_{X_1,x}$. We therefore denote it by $\Lambda_x$. 
\item[(3)] The graded pieces $$F_{r+1}/F_r$$ are independent of the chosen local parameter system of $\mathcal{O}_{X_1,x}$.  
\end{enumerate}
\end{lemma}
\begin{proof}
It suffices to show the proposition for two parameter systems $(f_1,f_2)$ and $(f_1,f_2')$ and $f_2=f_2'+\beta f_1$. We saw in Section \ref{first thickening} that $H^2_x(X_1,\Omega^1_{X_1})$ can be calculated locally as $\hat{H}^1(\text{Spec}\mathcal{O}_{X_1,x}\setminus\{x\},\Omega^1_{X_1})$ with respect to coverings of of $\text{Spec}\mathcal{O}_{X_1,x}\setminus\{x\}$. Now considering $\frac{df_2}{f_1^{n_1}f_2^{n_2}}\in \hat{H}^1(\text{Spec}\mathcal{O}_{X_1,x}\setminus\{x\},\Omega^1_{X_1})$ for the covering $D(f_1)\cup D(f_2)$ of $\text{Spec}\mathcal{O}_{X_1,x}\setminus\{x\}$, we can pass to the smaller covering $D(f_1)\cup D(f_2f_2')$ of $\text{Spec}\mathcal{O}_{X_1,x}\setminus\{x\}$. With respect to this covering 
 $$\frac{df_2}{f_1^{n_1}f_2^{n_2}}=\frac{df_2'}{f_1^{n_1}f_2'^{n_2}(1+\beta\frac{f_1}{f_2'})^{n_2}}+\frac{\beta df_1}{f_1^{n_1}f_2'^{n_2}(1+\beta\frac{f_1}{f_2'})^{n_2}}$$
is equivalent to 
$$\frac{(\sum_{n=0}^{\infty}(-\beta\frac{f_1}{f_2'})^n)^{n_2}df_2'}{f_1^{n_1}f_2'^{n_2}}+\frac{(\sum_{n=0}^{\infty}(-\beta\frac{f_1}{f_2'})^n)^{n_2}\beta df_1}{f_1^{n_1}f_2'^{n_2}}$$
since $\frac{1}{f_1^{n_1}f_2'^{n_2}(1+\beta\frac{f_1}{f_2'})^{n_2}}$ converges to $\frac{(\sum_{n=0}^{\infty}(-\beta\frac{f_1}{f_2'})^n)^{n_2}}{f_1^{n_1}f_2'^{n_2}}$ in $\mathcal{O}_{X_1,x}[\frac{1}{f_1f_2f_2'}]/\mathcal{O}_{X_1,x}[\frac{1}{f_2f_2'}]$, which lies again in $F_{n_1+n_2-1}$ considering it as an element of $\hat{H}^1(\text{Spec}\mathcal{O}_{X_1,x}\setminus\{x\},\Omega^1_{X_1})$ with respect to the covering $D(f_1)\cup D(f_2')$.
This proves $(1)$. $(2)$ and $(3)$ follow immediately.  
\end{proof} 

In order to prove Proposition \ref{surj2}, we need to prove key Lemma \ref{approximationlemma}. Its proof is inspired by the techniques of \cite{KeS16} from which we cite the following lemma:
\begin{lemma}(\cite[Lemma 10.2]{KeS16})\label{KeScitation} Let $X$ be a noetherian scheme, $E$ an effective Cartier divisor on $X$, and $A$ be an effective Cartier divisor on $E$. Let $\mathcal{F}$ be a coherent $\mathcal{O}_X$-module such that 
$$H^1(X,\mathcal{F}\otimes\mathcal{O}_X(-E))=H^1(E,\mathcal{F}|_E\otimes\mathcal{O}_E(-A))=0.$$
Then the restriction $res_A:H^0(X,\mathcal{F})\rightarrow H^0(A,\mathcal{F}|_A)$ is surjective.
\end{lemma}
By Lemma \ref{invariance} we can talk about the pole order of elements of $H^2_x(X_1,\Omega^1_{X_1})$ independent of the parameter system chosen. In particular, the following lemma makes sense:
\begin{keylemma}\label{approximationlemma} Every element $\gamma\in H^2_x(X_1,\Omega^1_{X_1})$ is equivalent to a sum of forms with simple poles in $H^2(X_1,\Omega^1_{X_1})$. More precisely, the image of $\gamma$ in $$H^2(X_1,\Omega^1_{X_1})\cong \mathrm{coker}(\oplus_{x\in X^{(1)}_1} H^{1}_x(X_1,\Omega^{1}_{X_1})\rightarrow \oplus_{x\in X^{(2)}_1} H^2_x(X_1,\Omega^{1}_{X_1}))$$
is equivalent to a sum $\sum_x\gamma'_x\in \oplus_{x\in X^{(2)}_1} H^2_x(X_1,\Omega^{1}_{X_1})$ with $\gamma'_x\in \Lambda_x\subset H^2_x(X_1,\Omega^{1}_{X_1})$ for every $x\in X^{(2)}_1$. Here $\Lambda_x\subset H^2_x(X_1,\Omega^{1}_{X_1})$ is the subgroup of forms with simple poles defined above.
\end{keylemma}

Key lemma \ref{approximationlemma} says in particular that there is a surjection $\oplus_{x\in X^{(2)}_1}\Lambda_x\twoheadrightarrow H^2(X_1,\Omega^1_{X_1})$. 

Before we start the proof, we introduce the following notation: Let $X$ be a regular Noetherian scheme and $Z$ an effective Cartier divisor on $X$. Let $C$ be a curve in $X$, i.e. an effective $1$-cycle on  $X$. Let 
$$(Z,C)_x:=\text{length}_{\mathcal{O}_{X,x}}(\mathcal{O}_{X,x}/I_Z+I_C)$$
be the intersection multiplicity of $Z$ and $C$ at $x$. We say that $Z$ and $C$ intersect transversally at $x$ if $(Z,C)_x=1$ and if $Z$ and $C$ are regular at $x$. If $Z$ and $C$ intersect transversally everywhere, we denote this by $Z\Cap C$.

\begin{proof}[Proof of Key lemma \ref{approximationlemma}] Without loss of generality we work with $\gamma=\frac{\alpha df_1}{f_1^{n_1}f_2^{n_2}}\in H^2_x(X_1,\Omega^1_{X_1})$. Let $D_1$ be a regular curve containing $x$ and $f_1'$ a local parameter of $D_1$ at $x$.  We consider $\frac{\alpha df_1}{f_1^{n_1}f_2^{n_2}}$ in $H^2_x(X_1,\Omega^1_{X_1})_{(f_1',f_2)}$. By Lemma \ref{invariance}, $\gamma$ is still in $F_{r+1}$ for $r+1=n_1+n_2$. We may assume that it is of the form $\frac{\alpha df_1'}{f_1'^{n_1}f_2^{n_2}}$.

Let $H\subset X_1$ be a hyperplane section and for an integer $d>0$ let $\mathcal{L}(d)=|dH|$ be the linear system of hypersurface sections of degree $d$. 

Now for $d\gg 0$ there exists an $F^1\in \mathcal{L}(d)$ such that 
\begin{enumerate}
\item $x\in F^1$,
\item $F^1\Cap D_1$ at any $y\in F^1\cap D_1$.
\end{enumerate}
For $d'$ sufficiently large relative to $d$, there exists an $F^2\in \mathcal{L}(d')$ such that 
\begin{enumerate}
\item $x\in F^2$,
\item $F^2\Cap D_1$ at any $y\in F^2\cap D_1$,
\item $F^1\cap F^2\cap D_1-x=\emptyset $.
\end{enumerate}
We choose $F^3,...,F^{n_2}$ analogously. Furthermore, we choose $F^{n_2}$ to be of sufficiently high degree so that
$$H^1(X_1,\Omega^1_{X_1}((n_1-1)D_1+F^1+...+F^{n_2})))=H^1(D,\Omega^1_{X_1}(n_1D_1+F^1+...+F^{n_2})|_D\otimes\mathcal{O}_D(-x))=0$$
holds by Serre vanishing. By Lemma \ref{KeScitation}, the last condition implies that the restriction map
$$H^0(X_1,\Omega^1_{X_1}(n_1D_1+F^1+...+F^{n_2}))\xrightarrow{res} H^0(x,\Omega^1_{X_1}(n_1D_1+F^1+...+F^{n_2})\otimes k(x))$$
is surjective. Let $y$ be the generic point of $D_1$.
By construction, the diagram 
$$\begin{xy}
  \xymatrix{
       H^0(X_1,\Omega^1_{X_1}(n_1D_1+F^1+...+F^{n_2}))  \ar[d]_{} \ar@{->>} [r]^-{res} &   \Omega^1_{X_1}(n_1D_1+F^1+...+F^{n_2})\otimes k(x) \ar[d]  \\
   H^1_y(X_1,\Omega^1_{X_1})  \ar[r]^-{d_x} & F_{r+1}/F_rH^2_x(X_1,\Omega^1_{X_1}) 
  }
\end{xy} $$
is commutative and $\frac{\alpha df_1}{f_1^{n_1}f_2^{n_2}}$ lies in $\Omega^1_{X_1}(n_1D_1+F^1+...+F^{n_2})\otimes k(x)$. Notice that the map on the right is well-defined.
This implies that there is a $\gamma\in H^1_y(X_1,\Omega^1_{X_1})$ such that $d_x(\gamma)=\frac{\alpha df_1}{f_1^{n_1}f_2^{n_2}}$. Furthermore for any $x'\in |D_1|-x$, the form $d_{x'}(\gamma)$ has at most simple poles in $f_{2}$ at $x'$, i.e. $d_{x'}(\gamma)\in F_{n_2+1}$. Now we apply the same construction to the form $d_{x'}(\gamma)$ which completes the proof. 
\end{proof} 

\begin{proof}[Proof of Proposition \ref{surj2}] Let $x\in X$ be a closed point and $X_{1,x}$ be the spectrum of the stalk of $\mathcal{O}_{X_1}$ in $x$. The $\check{\text{C}}$ech to derived functor spectral sequence 
$$E_2^{p,q}=\check{H}^p(\mathcal{U},\mathcal{H}^q(X,\mathcal{F}))\Rightarrow H^{p+q}(X,\mathcal{F})$$
induces an edge map
$$\check{H}^i(\mathcal{U},\mathcal{F})\rightarrow H^i(X,\mathcal{F}).$$
Since this edge map is functorial in $\mathcal{F}$, we get a commutative diagram
$$\begin{xy}
  \xymatrix{
        \check{H}^1(X_{1,x}-x,\Omega^1_{X_1}) \ar[d]_{\cong} \ar[r] & \check{H}^1(X_{1,x}-x,\mathcal{K}^M_{2,X_n})    \ar[d]  \\
     H^1(X_{1,x}-x,\Omega^1_{X_1})\cong H^2_x(X_1,\Omega^1_{X_1}) \ar[r]^{} &  H^1(X_{1,x}-x,\mathcal{K}^M_{2,X_n})\cong H^2_x(X_1,\mathcal{K}^M_{2,X_n})
  }
\end{xy} $$
for a closed point $x\in X_1$. We saw in Lemma \ref{cal2} that $ \check{H}^1(X_{1,x}-x,\Omega^1_{X_1})$ is generated by elements of the form $\frac{\alpha_1}{f_1^{n_1}}\frac{df_2}{f^{n_2}_2}+\frac{\alpha_2}{f_2^{n_2}}\frac{df_1}{f^{n_1}_1}$ for a local parameter system $(f_1,f_2)\in \mathcal{O}_{X_1,x}$ and by key Lemma \ref{approximationlemma} we may assume that it has simple poles. The point of the proof is that for forms with simple poles we can write down explicit lifts of these forms to $\CH_1(X)$.

Without loss of generality we consider the form $\alpha df_1/(f_1f_2)$. We lift $\alpha df_1/(f_1f_2)$ to 
$$\overline{V(\tilde{f_1},\tilde{f_2}+\pi^{n-1}\alpha)}-\overline{V(\tilde{f_1},\tilde{f_2})}\in\CH_1(X)$$
where $\tilde{f_1},\tilde{f_2}\in \O_{X,x}$ are lifts of $f_1$ and $f_2$.
We now show that this lift is mapped to $\text{exp}(\alpha df_1/(f_1f_2))=\{f_1,1+\pi^{n-1}\alpha/f_2\}\in H^2_x(X_1,\mathcal{K}^M_{2,X_n})$ by the restriction map. The map 
$$res:\CH_1(X)\rightarrow H^2(X,\mathcal{K}^M_{2,X}),$$
induced by the assumption of the Gersten conjecture, sends the cycle $\overline{V(\tilde{f_1},\tilde{f_2}+\pi^{n-1}\alpha)}-\overline{V(\tilde{f_1},\tilde{f_2})}$ to $(\{\tilde{f_1},\tilde{f_2}+\pi^{n-1}\alpha\},-\{\tilde{f_1},\tilde{f_2}\})$ 
in $$\check{H}^1(X_{\overline{V(\tilde{f_1},\tilde{f_2}+\pi^{n-1}\alpha)}}-\overline{V(\tilde{f_1},\tilde{f_2}+\pi^{n-1}\alpha)},\mathcal{K}^M_{2,X})\oplus \check{H}^1(X_{\overline{V(\tilde{f_1},\tilde{f_2})}}-\overline{V(\tilde{f_1},\tilde{f_2})},\mathcal{K}^M_{2,X}).$$
Finally, the restriction map $$H^2(X,\mathcal{K}^M_{2,X})\rightarrow H^2(X,\mathcal{K}^M_{2,X_n})$$
sends the tuple of $\check{\text{C}}$ech-cycles $(\{\tilde{f_1},\tilde{f_2}+\pi^{n-1}\alpha\},-\{\tilde{f_1},\tilde{f_2}\})$ to $\{\tilde{f_1},1+\pi^{n-1}\alpha/\tilde{f_2}\}\equiv \{f_1,1+\pi^{n-1}\alpha/f_2\}\in \check{H}^1(X_1-x,\mathcal{K}^M_{2,X_n})$. In sum this shows in particular that $$\ker[H^2(X_1,\mathcal{K}^M_{2,X_n})\rightarrow H^2(X_1,\mathcal{K}^M_{2,X_{n-1}})]$$ is in the image of $\CH_1(X)$.

The surjectivity of $\CH_1(X)\rightarrow H^2(X_n,\mathcal{K}^M_{2,X_n})$ now follows by induction: For $n=1$ this is just the surjectivity of $\CH_1(X)\rightarrow \CH_0(X_1)$. For $n>1$, let $\alpha\in H^2(X_1,\mathcal{K}^M_{2,X_n})$. Let $\alpha_{n-1}$ be the image of $\alpha$ in $H^2(X_1,\mathcal{K}^M_{2,X_{n-1}})$. By assumption there is a cycle $Z\in \CH_1(X)$ mapping to $\alpha_{n-1}$. Denote the image of $Z$ in $H^2(X_n,\mathcal{K}^M_{2,X_n})$ by $\alpha_Z$. Now $\alpha-\alpha_Z$ is in the kernel of $H^2(X_1,\mathcal{K}^M_{2,X_n})\rightarrow H^2(X_1,\mathcal{K}^M_{2,X_{n-1}})$ and by the above construction lifts to an element $Z'\in \CH_1(X)$. Now $Z'-Z$ maps to $\alpha$. 
\end{proof} 

\begin{corollary}\label{cormain}
Let $X$ be as in Proposition \ref{surj2} but of arbitrary relative dimension $d$ over $A$. Then, assuming the Gersten conjecture for the Milnor K-sheaf $\mathcal{K}^M_{*,X}$, the map $$res_{X_n}:\CH_1(X)\rightarrow H^d(X_1,\mathcal{K}^M_{d,X_n})$$ is surjective. In particular the map 
$res:\CH_1(X)\rightarrow "\mathrm{lim}_n" H^d(X_1,\mathcal{K}^M_{d,X_n})$
is an epimorphism in $\text{pro-}\text{Ab}$. 
\end{corollary}

\begin{proof} We may assume that $d\geq 3$.
Let $\alpha\in H^d(X_1,\mathcal{K}^M_{d,X_n})$. By the Bertini theorems of Bloch, Altman-Kleiman and Poonen (see \cite{Bl71}, \cite{AK79}, \cite{Po08}) we may find a smooth hypersurface section $Y_1$ of $X_1$ containing the the support of the image of $\alpha$ in $\CH_0(X_1)$ under the restriction map $H^d(X_1,\mathcal{K}^M_{d,X_n})\rightarrow H^{d}(X_1,\mathcal{K}^M_{d,X_1})\cong \CH_0(X_1)$. Furthermore we may choose $Y_1$ to be of very high degree in which case it lifts to a smooth projective subscheme $Y$ of $X$ over $A$ by Serre vanishing (see for example \cite[Lem. 2.6]{Lu16}). Then by Corollary \ref{cord} the element $\alpha$ is in the image of the map $H^{d-1}(Y_1,\mathcal{K}^M_{d-1,Y_n})\r H^{d}(X_1,\mathcal{K}^M_{d,X_n})$. It follows from the commutative diagram 
$$\begin{xy}
  \xymatrix{
        \CH_1(X) \ar[r]^-{res_{X_n}}   & H^{d}(X_1,\mathcal{K}^M_{d,X_n})    \\
       \CH_1(Y)   \ar[r]^-{res_{Y_n}}  \ar[u]  &   H^{d-1}(Y_1,\mathcal{K}^M_{d-1,Y_n}) \ar[u] 
  }
\end{xy} $$
(see Lemma \ref{gysinmap}) and Proposition \ref{surj2} and induction on $d$ that $\alpha$ is in the image of $res_{X_n}$.
\end{proof}

\begin{remark}
It can be shown that a more general version of Key lemma \ref{approximationlemma} holds for $X_1$ of arbitrary dimension $d$. Indeed, if $Y_1:=Y_1^{(d-2)}\subset Y_1^{(d-2)}\subset...Y_1^{(1)}\subset X_1$ is a sequence of smooth hypersurface sections of high degree of $X_1$ defined locally by a sequence $(f_1,...,f_{d-2})$ of regular parameters, then by Proposition \ref{lefschetz} there is a surjection 
$$gys:H^{2}(Y_1,\Omega^{1}_{Y_1})\rightarrow H^d(X_1,\Omega^{d-1}_{X_1}).$$ 
This gives a surjection from
$$\mathrm{coker}(\oplus_{x\in Y^{(1)}_1} H^{1}_x(Y_1,\Omega^{1}_{Y_1})\rightarrow \oplus_{x\in Y^{(2)}_1} H^2_x(Y_1,\Omega^{1}_{Y_1}))$$
onto
$$\mathrm{coker}(\oplus_{x\in X^{(d-1)}_1} H^{d-1}_x(X_1,\Omega^{d-1}_{X_1})\rightarrow \oplus_{x\in X^{(d)}_1} H^d_x(X_1,\Omega^{d-1}_{X_1})).$$
For $x\in Y_1\subset X_1$ the map 
$$H^2_x(Y_1,\Omega^{1}_{Y_1})\r H^d_x(X_1,\Omega^{d-1}_{X_1})$$
inducing the above map is given by 
$$\omega\mapsto \omega\wedge \mathrm{dlog}f_{1}\wedge...\wedge \mathrm{dlog}f_{d-2}.$$
By Key lemma \ref{approximationlemma} this implies that $H^d(X_1,\Omega^{d-1}_{X_1})$ is generated by forms with simple poles. The latter means forms as in Lemma \ref{cal2} such that $n_1=...=n_d=1$.
\end{remark}

\section{Open problems}\label{sectionopenproblems}
Let $A$ be an excellent henselian discrete valuation ring with uniformising parameter $\pi$ and residue field $k$. Let $X$ be a smooth projective scheme over Spec$(A)$ of relative dimension $d$. Let $X_n:=X\times_A A/(\pi^n)$, i.e. $X_1$ is the special fiber and the $X_n$ are the respective thickenings of $X_1$. Let us recall the following conjecture by Kerz, Esnault and Wittenberg:
\begin{conj}\label{conj1}\cite[Sec. 10]{KEW16} Assume the Gersten conjecture for the Milnor K-sheaf $\mathcal{K}^M_{*,X}$. The map
$$res^{/p^r}: \CH^{d}(X)/p^r \r "\mathrm{lim}_n" H^{d}(X_1,\mathcal{K}^M_{d,X_n}/p^r)$$ is an isomorphism in pro-Ab if $\mathrm{ch}(\mathrm{Quot}(A))=0$ and if $k$ is perfect of characteristic $p>0$ .
\end{conj}
In Section \ref{secsurj} we showed that $res$ is surjective. For the fact that the surjectivity of $res$ implies that of $res^{/p^r}$ see \cite[Cor. 5.4]{Lu17'} or note that the same arguments as above go through mod $p^r$ since tensoring with $\Z/p^r\Z$ is right exact. In Section \ref{reldim1} we showed that Conjecture \ref{conj1} holds (under some assumptions on the base) if $d=1$. In order to explain how the method of this article relates to the injectivity of $res^{/p^r}$ in arbitrary dimension, we need to first recall the main ideas of the proof of Kerz, Esnault and Wittenberg of Conjecture \ref{conj1} for $(n,p)=1$, i.e. the following theorem:
\begin{theorem}
Let $p$ be the exponential characteristic of $k$ and let $(n,p)=1$. Then the map
$$res_{X_1}^{/n}: \CH^{1}(X)/n\r \CH_0(X_1)/n$$
is an isomorphism.
\end{theorem}
In the following we simply write $res$ instead of $res_{X_1}^{/n}$. This theorem was first proved by Saito and Sato for $k$ finite or separably closed in \cite{SS10} and then for arbitrary residue fields by Kerz, Esnault and Wittenberg using an idea of Bloch in \cite{KEW16}. The strategy of \cite{KEW16} is to construct a well-defined inverse map to $res$ and can be summed up in the following diagram:
\[\xymatrix{
  \CH_1(X)/n \ar[r]^{res} & \CH_0(X)/n \ar@/_15pt/@{->}[l]_{res^{-1}}  \\
      Z_1^g(X)  \ar[u] \ar[r]_{}  &   Z_0(X_0) \ar[u]^{} \ar[lu]_{lift}
  }\]
Here $Z_1^g(X)$ are one-cycles in good position, i.e. horizontal one-cycles. The map $lift$ is a priori not well-defined. There are many ways to lift a zero-cycle to a horizontal one-cycle (see \cite[Sec. 4]{KEW16}). However it can be shown that mod $n$ such a map exists and that it factors through rational equivalence. This leads to the following conjecture in the mod $p$ case:
\begin{conj}\label{conjinversemap} (Kerz)
There is a filtration $..\subset F_2\subset F_1\subset \CH_1(X)$ and a well defined map
$$\CH_1(X)/F_n \leftarrow H^{d}(X_1,\mathcal{K}^M_{d,X_n})$$
which is inverse to $res_{X_n}$. Furthermore $"\lim_n"F_n\otimes \Z/p^r\Z=0$. \end{conj}
The proof of Proposition \ref{surj2} is a first step towards constructing such an inverse map. It can be summed up in the following diagram:
$$\begin{xy}
  \xymatrix{
   & \Lambda_x  \ar[dr]_{} \ar@{.>}[ld]_{lift} & \\
     \CH_1(X)  \ar[d]_{} & & H^2_x(X_1,\Omega_{X_1}^1)   \ar[d]  \\
   H^2(X,\mathcal{K}^M_{2,X})          \ar[rr]^{}  &  &   H^2(X_1,\mathcal{K}^M_{2,X_2}) \ar[d] \\
   & & H^2(X_1,\mathcal{K}^M_{2,X_1})
  }
\end{xy} $$
We expect that the lifts constructed in the course of the proof of Proposition \ref{surj2} are unique up to the conjectured filtration. In the beginning of Section \ref{secsurj} we defined such a filtration for a smooth projective scheme $C$ of relative dimension $1$ over $A$: let $F_n$ be the subgroup of $\CH_1(C)$ generated by all cycles $Z$ vanishing on $C_n$, i.e. $Z|_{C_n}=0$. Furthermore we showed that there is a well-defined map $\gamma_x:{O}_{X_1,x}[\frac{1}{f}]/\mathcal{O}_{X_1,x}\rightarrow \CH_1(X)/F_n$.
Now let $S$ be the set smooth projective schemes $C$ of relative dimension $1$ equipped with a projective $\Spec A$-morphism $i_C:C\r X$. A possible candidate for the filtration in Conjecture \ref{conjinversemap} is the following:
$$F_n:=<i_{C*}F_n|C\in S >.$$

For a similar conjecture see \cite[Conj. 5.7]{Lu17'}. 

We now turn to the relationship with the following question which we recall from the introduction:
\begin{quest}\label{questctopenprob}(\cite[Question 1.4(g)]{Co95})
Let $X_K$ be a smooth projective and connected variety over a $p$-adic field $K$. 
Let $A_0(X_K)$ denote the kernel of the degree map $\mathrm{deg}:\CH_0(X_K)\r \Z$ and $D(X_K)$ its maximal divisible subgroup. Is 
$$A_0(X_K)/D(X_K)\cong \Z^n_p\oplus (\mathrm{finite\; group})$$
for some $n\in \mathbb{N}$?
\end{quest}
Let $A$ be the ring of integers in $K$. Assume that $X_K$ has a smooth and projective model  $X$ of relative dimension $d$ over $A$. For a smooth projective (over $A$) subscheme of codimension one $Y\subset X$ we expect that the map 
$$"\mathrm{lim}_n" H^{d-1}(Y_1,\mathcal{K}^M_{d-1,Y_n}/p^r)\r "\mathrm{lim}_n" H^d(X_1,\mathcal{K}^M_{d,X_n}/p^r)$$
constructed in Lemma \ref{gysinmap} is an isomorphism for $d\geq 3$ and surjective for $d=2$ if we choose $Y$ to be of high degree.
For $n=1$, this follows from class field theory and standard Lefschetz theorems for the \'etale fundamental group. In order to prove this for arbitrary $n$ using Proposition \ref{exactsequwitt}, the Lefschetz theorem of Section \ref{lefschetzsection}, Proposition \ref{lefschetz}, needs to be improved by one degree:
\begin{quest}
Is there a context in which the statement of Proposition \ref{lefschetz} can be improved by one degree? 
\end{quest}
A positive answer to this question and the injectivity of $res^{/p^r}$ for arbitrary dimension would imply that the natural map
$$\CH_1(Y)/p^r\r \CH_1(X)/p^r$$
is bijective for $X$ of relative dimension $d\geq 3$ and surjective for $X$ of relative dimension $d=2$. Since $\CH_1(Y)\r \CH_0(Y_K)$ and $\CH_1(X)\r \CH_0(X_K)$ are surjective, this would would imply that the map
$$\CH_0(Y_K)/p^r\r \CH_0(X_K)/p^r$$
surjective for $\mathrm{dim}X_K=d\geq 2$. 
For a smooth projective curve $C_K$ over $K$ we know that $$A_0(C_K)\cong \Z_p^m\oplus (\text{finite group})$$ for some $m\in\mathbb{N}$ (see \cite{Ma55}). This implies, under the above assumptions, the same result for the $p$-completion of $\CH_1(X_K)$ and therefore a positive answer to Question \ref{questctopenprob}.

The corresponding weak Lefschetz theorem for $n$ prime to $p$ saying that the map
$$\CH_1(Y)/n\r \CH_1(X)/n$$
is bijective for $d\geq 3$ and surjective for $d=2$ is proved in \cite[Cor. 9.6]{SS10}.

\bibliographystyle{siam}
\bibliography{Bibliografie} 
\end{document}